\documentclass[letterpaper, 10 pt, conference]{ieeeconf}  % Comment this line out if you need a4paper

\IEEEoverridecommandlockouts                              % This command is only needed if 
\usepackage{graphicx,algorithm,amsthm,amsmath,amsfonts,verbatim,mathtools,enumerate,bm,hyperref,cite}
\usepackage[dvipsnames]{xcolor}
\usepackage{tikz, subcaption, pgfplots}
\usepackage{algpseudocode}
\pgfplotsset{compat=1.18}
\usepgfplotslibrary{statistics}

\allowdisplaybreaks

\newcommand{\norm}[1]{\left\lVert#1\right\rVert}
\newcommand{\mnorm}[1]{{\left\vert\kern-0.25ex\left\vert\kern-0.25ex\left\vert #1 
    \right\vert\kern-0.25ex\right\vert\kern-0.25ex\right\vert}}

\newtheorem{theorem}{Theorem}

\newcommand{\eg}{{\it e.g.}}

\title{\LARGE \bf Nonlinear Trajectory Optimization Models for Uncrewed Aerial Vehicles with Mobile Charging Support
}

\author{Minsen Yuan, Amanuel Adane, James Humann, and Yue Yu% <-this % stops a space
\thanks{M. Yuan and Y. Yu are with the Department of Aerospace Engineering and Mechanics, University of Minnesota, Minneapolis, MN 55455, USA ({\tt\small \{yuan0450, yuey\}@umn.edu}). A. Adane is with the Department of Computer Science, Cornell University, Ithaca, NY 14853, USA ({\tt\small aa2888@cornell.edu}). J. Humann is with DEVCOM Army Research Laboratory ({\tt\small james.d.humann.civ@army.mil}). Y. Yu would like to thank Samet Uzun for helpful early discussions.}
}

\definecolor{NLP_color}{rgb}{1,0,0}
\definecolor{MINLP_color}{rgb}{0.301,0.745,0.933}
\definecolor{MINLP_color}{rgb}{0,0.447,0.741}

\begin{document}

\maketitle
\thispagestyle{empty}
\pagestyle{empty}

%%%%%%%%%%%%%%%%%%%%%%%%%%%%%%%%%%%%%%%%%%%%%%%%%%%%%%%%%%%%%%%%%%%%%%%%%%%%%%%%
\begin{abstract}

Supporting Uncrewed Aerial Vehicles (UAVs) with mobile charging stations enables persistent UAV autonomy in infrastructure-sparse environments. In this setting,
trajectory optimization for UAVs is challenging because it couples task scheduling with when and where to recharge, as well as terrain-access constraints on where charging is available.
We propose a smooth nonlinear trajectory optimization model for UAV with mobile charging support.
Compared with existing results, the proposed model allows nonlinear charging dynamics mode via a unified battery dynamics model with disjunctive constraints on the time allocated to each mode. Furthermore, it provides 
smooth approximations of the disjunctive constraints with bounded approximation errors. By avoiding integer variables, these approximations enable efficient solution using smooth nonlinear optimization algorithms. We evaluate the proposed model on UAV missions with multiple spatially distributed tasks, nonlinear constant-current--constant-voltage charging dynamics, and terrain-access constraints on mobile charging support. Compared with mixed-integer nonlinear programs, the proposed model provides high-quality approximate solutions while reducing the computation time by orders of magnitude.

% Energy-sharing UAV-UGV systems extend the endurance of Uncrewed Aerial Vehicles (UAVs) by leveraging Uncrewed Ground Vehicles (UGVs) as mobile charging stations, enabling persistent autonomy in infrastructure-sparse environments. Trajectory optimization for these systems is often challenging due to UGVs' terrain access constraints and the discrete nature of task scheduling. We propose a smooth nonlinear program model for the joint trajectory optimization for these systems. Unlike existing models, the proposed model allows smooth parameterization of UGVs' terrain access constraints and supports partial UAV recharging. Further, it introduces a smooth approximation of disjunctive constraints that eliminates the need for computationally expensive integer programming and enables efficient solutions via nonlinear programming algorithms. We demonstrate the proposed model on a one-UAV-one-UGV system with multiple task locations. Compared with mixed-integer nonlinear programs, this model reduces the computation time by orders of magnitude.

\end{abstract}

%%%%%%%%%%%%%%%%%%%%%%%%%%%%%%%%%%%%%%%%%%%%%%%%%%%%%%%%%%%%%%%%%%%%%%%%%%%%%%%%
\section{Introduction}

%{\color{blue}Support UAV w/ mobile charging stations} Collaborative UAV-UGV systems provide a powerful platform to integrate the heterogeneous strengths of aerial and ground robots via \emph{energy-sharing}. In these systems, Uncrewed Ground Vehicles (UGVs) serve as mobile docking and charging stations that effectively extended the endurance of Uncrewed Aerial Vehicles (UAVs) via \emph{sharing} energy stored on UGVs. By combining UGVs' superior payload capacity together with UAVs' agility and elevated perspective, these systems enable a wide range of robotic applications in remote and challenging environments where access to traditional energy infrastructure is sparse. Examples of these applications include mapping, disaster response, surveillance, and precision agriculture \cite{ding2021review,chai2024cooperative,munasinghe2024comprehensive}. 

Supporting Uncrewed Aerial Vehicles (UAVs) with mobile charging stations---such as ground vehicles equipped with wireless charging capabilities---provides an effective approach to extending flight endurance and operational range. By enabling recharging between flights, such support reduces the need for UAVs to repeatedly return to fixed charging infrastructure. Mobile charging support is particularly valuable for long-duration missions with multiple spatially distributed tasks in remote or infrastructure-limited environments. Example applications of UAVs supported by mobile charging support include mapping, disaster response, surveillance, and precision agriculture \cite{ding2021review,chai2024cooperative,munasinghe2024comprehensive}. 

%{\color{blue}Planing UAV trajectories w/ mobile charging support} Coordinating the trajectories for energy-sharing UAV-UGV systems presents a unique challenge, especially in missions with multiple tasks locations. This difficulty is twofold. First, The UGV's terrain access constraints restrict the UGVs to a known road network. Second, scheduling the the order in which the UAVs and UGVs completes their individual tasks often require optimizing discrete decisions. When these discrete decisions are coupled with the terrain access constraints, the resulting optimization is often nonconvex and computationally expensive to solve. 

Planning trajectories for UAVs supported by mobile charging support introduces a challenging optimization problem, especially for missions with multiple spatially distributed tasks. The UAV must determine not only the order in which tasks are completed, but also when and where to access mobile charging stations. Moreover, mobile charging stations are often subject to terrain-access constraints that are more restrictive than those of UAVs, hence limiting where charging can occur. The resulting trajectory optimization problem for the UAV is often nonconvex and computationally expensive to solve.

One approach to optimize trajectories for energy-constrained UAV supported by mobile charging stations is to jointly optimize their trajectories.
One example is to use mixed-integer linear programming (MILP), which computes optimal UAV and mobile charging station trajectories via optimization over integer variables~\cite{wang2019,booth2020target,manyam2019}. Since MILP is often computationally expensive, there have been several strategies to simplify the problem under stronger assumptions. For example, if the potential rendezvous locations are known, a genetic algorithm provides an efficient alternative to MILP~\cite{eker2025}. If the order of UAV and mobile charging station actions is known, then optimizing the trajectories of the UAV and the mobile charging station reduces to a convex second-order cone program~\cite{diller2026}. Another approach discretizes the UAV battery level and formulates the joint routing and recharging problem as a generalized traveling salesperson problem~\cite{yu2018algorithms}. A more recent direction is to use deep reinforcement learning methods to train encoder-decoder-based transformer networks that generate the UAV-mobile charging station trajectories~\cite{mondal2024,li2025}. These trajectories can not only minimize mission time but also account for stochastic battery usage by constraining the risk of mid-mission battery depletion~\cite{mondal2025}.   

To address the difficulty of coupling between UAVs and mobile charging stations, an alternative approach is to decompose their trajectory planning into separate stages.
One approach is to first determine the UAV trajectory, then optimize the mobile charging station trajectories to support the UAV~\cite{mathew2015}. Another approach is to first determine the mobile charging station trajectories that provide a set of feasible rendezvous locations to support UAV recharging, then plan the UAV trajectory based on these locations~\cite{maini2015, mondal2025cooperative}. Early results use a greedy algorithm to generate the rendezvous locations, then compute the UAV trajectory accordingly~\cite{maini2015,maini2019}. More recent results augment this approach with an asynchronous team framework \cite{ramasamy2024} and task allocation heuristics based on minimum set covering~\cite{mondal2025cooperative}. 
A partition-based approach computes mobile charging station tours between regional release points and separate UAV tours within assigned subregions~\cite{lin2022robust}.
A reinforcement learning approach similarly clusters customers and then solves separate truck and drone routing subproblems~\cite{wu2023}.
Another approach is to decouple the planning of UAV and mobile charging station trajectories using reachable-set constraints~\cite{kim2024decoupled}. In all such cases, the resulting trajectory quality depends on how the planning problem is decomposed between the UAV and mobile charging stations.

% One approach to optimize energy-constrainted UAV trajectories with mobile charging stations is to jointly optimize their trajectories. On approach for such optimization is mixed integer linear programs (MILP), which computes an optimal UAV and mobile charging station trajectory via optimization over integer variables \cite{wang2019,manyam2019}. Since MILP is often computationally expensive, there have been several strategies to simplify the problem under stronger assumptions. For example, if the potential rendezvous locations are known, genetic algorithm provides an efficient alternative to MILP \cite{eker2025}. If the order of UAV and mobile charging station actions are known, then optimizing UAV-mobile charging station trajectories reduces to a convex second-order cone program \cite{diller2026}. A more recent direction is deep reinforcement learning methods to train encoder-decoder based transformer networks that generates the UAV-mobile charging station trajectories \cite{wu2023,mondal2024,li2025}. These trajectories can not only minimize mission time but also account for stochastic battery usage by constraining on the risk of mid-mission battery depletion \cite{mondal2025}.   

% To address the difficulty of coupling between UAVs and mobile charging stations, an alternate approach is via the \emph{two-echelon routing problem}, where the trajectories of the UAV and mobile charging stations are computed in two distinct phases. 

% {\color{blue} Limitations: nonlinear battery charging/discharging dynamics, integer variables}

In the existing work on UAV trajectory optimization with mobile charging support, two commonly used modeling choices can limit the applicability and computational scalability of these models.
First, many formulations use simplified battery charging models, such as full recharge before takeoff~\cite{maini2019,lin2022robust,mondal2024,mondal2025cooperative,ramasamy2024,mondal2025,eker2025} or a constant charging rate~\cite{yu2018algorithms}. These simplified models do not capture the flexibility of wireless charging or the nonlinear charging dynamics, such as those of lithium-ion batteries~\cite{hussein2011review}.
Second, many existing approaches rely on MILP to model discrete decisions, such as switching between charging and discharging modes, causing the solution time to grow rapidly with the number of integer variables~\cite{wang2019,booth2020target,manyam2019,maini2019,mondal2025cooperative,mathew2015,ramasamy2024}.

To address these limitations, we propose a novel \emph{nonlinear} trajectory optimization model for energy-constrained UAV systems with mobile charging support. 
First, we propose a battery-dynamics model that allows general nonlinear charging and discharging behavior. Rather than switching between separate dynamics, this model accounts for charging and discharging modes jointly through a single nonlinear dynamics equation, with disjunctive constraints on the time allocated to each mode. Second, to avoid the integer variables commonly used to model disjunctive constraints, we reformulate these constraints as equivalent nonsmooth constraints defined by pointwise minimum functions. We develop a parameterized smooth approximation based on \(\ell_p\)-norms. We provide rigorous bounds on the approximation error in both function value and gradient direction as the smoothing parameter varies. We demonstrate the proposed model through a UAV trajectory-optimization problem with mobile charging support. This problem includes nonlinear constant-current--constant-voltage charging dynamics for the UAV, multiple spatially distributed tasks for the UAV, terrain-access constraints on where charging can occur, and distinct velocity bounds for the UAV during charging and discharging. Compared with mixed-integer nonlinear programming, the proposed model---which solves a nonlinear program instead---reduces solution time from hours to minutes while maintaining robust success rates in numerical experiments.

\section{Trajectory Optimization Models for UAV systems with mobile charging support}\label{sec: Models}
%UAV Trajectory Optimization with Disjunctive Constraints

We formulate the trajectory optimization problem for UAV systems with mobile charging support as a nonlinear optimization problem with disjunctive constraints. We first introduce the decision variables that parameterize the UAV trajectory, then present the objective function and the physical and operational constraints.

\subsection{Trajectory Variables} \label{subsection: Trajectory_Variables}
Let \(N\in\mathbb{N}\) denote the total number of time stamps along the UAV trajectory. To simplify the notation, we let \([k] \coloneqq \{1, 2, \ldots, k\}\) for any \(k \in \mathbb{N}\).
\paragraph{Position Variables}
% We parameterize the UAV trajectory using its planar position. We let
% \(r_k\in\mathbb{R}^2\) denote the projection of the UAV's three-dimensional position onto the \(xy\)-plane at the \(k\)-th time stamp, for all \(k\in[N]\).
We parameterize the UAV trajectory by its position \(r_k\in\mathbb{R}^d\) at the \(k\)-th time stamp, for all \(k\in[N]\), where \(d\) denotes the spatial dimension.
\paragraph{Energy Variables}
% We characterize the UAV battery level in terms of the remaining flight time supported by the current battery energy. We let
% \(e_k\in\mathbb{R}_{\geq 0}\) denote the remaining flight time of the UAV at the \(k\)-th time stamp, for all \(k\in[N]\).
We characterize the UAV battery level as a fraction of its full capacity. We let \(e_k\in \mathbb{R}_{\geq 0}\) denote the normalized battery level at the \(k\)-th time stamp, for all \(k\in[N]\).

\paragraph{Time Variables}
We let \(s_k\in\mathbb{R}_{\geq 0}\) denote the duration between the \(k\)-th and the \((k+1)\)-th time stamps, for all \(k\in[N-1]\).
We further let \(c_k\in\mathbb{R}_{\geq 0}\) denote the charging duration allocated to the UAV at a mobile charging station between the \(k\)-th and the \((k+1)\)-th time stamps, for all \(k\in[N-1]\).

\subsection{Objective Function}
We choose the total mission time as the objective function. Minimizing total time directly captures the operational efficiency of the UAV system with mobile charging stations, encouraging timely completion of all required tasks while implicitly balancing travel and charging times. Since the time durations between consecutive time stamps are optimization variables, the total mission time is given by
\begin{equation} \label{eqn: obj}
   \textstyle  \sum_{k=1}^{N-1} s_k.
\end{equation}

\subsection{Trajectory Constraints}
We consider two types of constraints for the UAV trajectories: smooth constraints and disjunctive constraints. 
\paragraph{Smooth Constraints} We let \(\overline{r}_0\in\mathbb{R}^d\) and \(\overline{r}_f\in\mathbb{R}^d\) denote the initial and final position of the UAV trajectory, respectively. We consider the following initial and final constraints
\begin{equation}
    r_1=\overline{r}_0, \enskip r_N=\overline{r}_f.
\end{equation}
In addition, we consider the following constraint on the speed of the UAV
\begin{equation} \label{eqn: speed_const}
    \norm{r_{k+1} - r_k}_2 \leq v_{\texttt{d}} s_k,
\end{equation}
for all \( k \in [N - 1] \), where \(v_{\texttt{d}}\in\mathbb{R}_{>0}\) is the maximum UAV speed. 
% Note that the UAV position is three-dimensional, whereas the above constraint accounts only for planar motion. This is because we assume that the UAV operates at a constant altitude, and that the time required for vertical motion (\emph{e.g.}, during takeoff and landing) is negligible compared with the time spent in horizontal motion.

We impose the following bounds on the UAV battery level:
\begin{equation}
    e_{\min} \leq e_k \leq e_{\max},
    \qquad k \in [N],
\end{equation}
where \(0 \leq e_{\min} < e_{\max}\) denote the minimum and maximum allowable battery levels, respectively.
Let \(f_{\texttt{c}}(e,t)\) and \(f_{\texttt{d}}(e,t)\) denote the battery levels obtained after charging and discharging, respectively, for a duration \(t\) starting from battery level \(e\). We assume that
\begin{equation}
    \label{eqn:zero_duration_identity}
    f_{\texttt{c}}(e, 0) = f_{\texttt{d}}(e, 0) = e
\end{equation}
for every admissible battery level \(e\). Using the segment duration \(s_k\) and charging duration \(c_k\), we define the unified battery dynamics as
\begin{equation}
\begin{aligned}
    e_{k+1}
    &= f(e_k,s_k,c_k) \\
    &\coloneqq f_{\texttt{c}}(e_k,c_k)
    + f_{\texttt{d}}(e_k,s_k-c_k) - e_k,
\end{aligned}
\label{eq:battery_dynamics}
\end{equation}
for all \(k \in [N-1]\). In particular, \eqref{eqn:zero_duration_identity} implies that
\begin{equation}\label{eq:batteryDynamics}
    f(e_k,s_k,c_k) =
    \begin{cases}
        f_{\texttt{d}}(e_k,s_k), & c_k = 0, \\
        f_{\texttt{c}}(e_k,s_k), & c_k = s_k.
    \end{cases}
\end{equation}
Thus, \eqref{eq:battery_dynamics} recovers the discharging dynamics when \(c_k = 0\) and the charging dynamics when \(c_k = s_k\). This unified representation accommodates nonlinear charging and discharging maps satisfying \eqref{eqn:zero_duration_identity}. We enforce charging–discharging selection and the additional conditions required for charging through the disjunctive constraints introduced below.

\paragraph{Disjunctive Constraints} We consider the case where the UAV must visit a set of task locations along its trajectory (e.g., to monitor areas of interest). 
Let \(\tau\in\mathbb{N}\) denote the total number of task locations, with positions \(a_1,\ldots,a_\tau\in\mathbb{R}^d\).
We consider the following disjunctive constraints:
\begin{equation}
    \bigvee_{k \in [N]} \left\{r_k = a_i\right\},
\end{equation}
for all \( i \in [\tau]  \). 
These constraints ensure that the UAV visits each task point at least once along its trajectory.

Furthermore, the UAV can discharge its battery at any location, but it can charge only within one of \(J\) prescribed charging regions. We assume that a mobile charging station can serve the UAV within each region. 
% During charging, the UAV must remain on the mobile charging station and within the same charging region, traveling at a speed no greater than the station’s maximum speed. 
The constraints that apply when the UAV charges on a mobile charging station can, without loss of generality, be modeled as \(g_j(r_k,r_{k+1},s_k)\leq\bm{0}^{l}\), where \(g_j:\mathbb{R}^d\times\mathbb{R}^d\times\mathbb{R}\to\mathbb{R}^{l}\) is continuously differentiable for each \(j\in[J]\), \(l\) denotes the number of scalar constraints, and the inequality is interpreted componentwise.

% Consequently, the UAV's displacement over the charging segment cannot exceed \(v_{\texttt{c}} s_k\), where \(v_{\texttt{c}}\) denotes the maximum speed of the mobile charging station.
Thus, we impose the following disjunctive constraint for charging and discharging:

% Furthermore, the UAV can discharge its battery at any location, but we allow it to charge only within one of the prescribed charging regions characterized by \(\varphi_j(r)\leq \bm{0}^l\), where $\varphi_j:\mathbb{R}^d\to\mathbb{R}^l$ is continuously differentiable for each $j\in[J]$, and $J$ denotes the number of charging regions. We assume that a mobile charging station can serve the UAV within each region. During charging, the UAV remains on the mobile charging station and within the same charging region. Consequently, the UAV's displacement over the charging segment cannot exceed \(v_{\texttt{c}} s_k\), where \(v_{\texttt{c}}\) denotes the maximum speed of the mobile charging station.
% Thus, we impose the following disjunctive constraint for charging and discharging:
\begin{equation}
\label{eq:charging_disjunction}
\left(
\bigvee_{j \in [J]}
\left\{
% \begin{aligned}
%     c_k &= s_k,\\
%     \varphi_j(r_k) &\leq 0,\\
%     \varphi_j(r_{k+1}) &\leq 0,\\
%     \norm{r_{k+1}-r_k}_2
%     &\leq v_{\texttt{c}} s_k
% \end{aligned}
\begin{aligned}
    c_k &= s_k,\\
    g_j(r_k, r_{k+1}, s_k) &\leq \bm{0}^l
\end{aligned}
\right\}
\right) \lor \; \{c_k = 0\},
\end{equation}
for all \(k \in [N-1]\). 
% When charging occurs, the UAV speed constraints~\eqref{eqn: speed_const} are automatically satisfied, since $v_{\texttt{d}} > v_{\texttt{c}}$.
These constraints require the UAV to either discharge its battery or charge on a mobile charging station within the same prescribed charging region between consecutive time stamps.

\subsection{Trajectory Optimization Problem}
The decision variables introduced in Section~\ref{subsection: Trajectory_Variables} are
\begin{equation}\label{eqn: base var}
\left\{
        r_k, 
        e_k
    \right\}_{k=1}^N \cup \left\{
        s_k, c_k
    \right\}_{k=1}^{N-1}.
\end{equation}
The smooth constraints are summarized as follows:
\begin{equation}\label{eqn: smooth constr}
    \begin{aligned}
&r_1=\overline{r}_0, \enskip r_N = \overline{r}_f, \enskip e_1=e_{\max},\\
&\norm{r_{k+1} - r_k}_2 \leq v_{\texttt{d}} s_k,\enskip  k \in [N - 1],\\
&  e_{\min}\leq e_{k+1}\leq e_{\max},\enskip s_{\min} \leq s_k \le s_{\max},\enskip k \in [N - 1],\\
& e_{k+1} = f(e_k, s_k, c_k), \enskip k \in [N - 1].
    \end{aligned}
\end{equation}
Note that the velocity constraints admit equivalent smooth representations by squaring both sides. The disjunctive constraints are given by
\begin{equation}\label{eqn: disjunctive constr}
    \begin{aligned}
    &\bigvee_{k \in [N]} \left\{r_k = a_i\right\}, \enskip i \in [\tau], \\
    & \left(
\bigvee_{j \in [J]}
\left\{
\begin{aligned}
    c_k &= s_k,\\
    g_j(r_k, r_{k+1}, s_k) &\leq \bm{0}^l
\end{aligned}
\right\}
\right) \lor \; \{c_k = 0\}, \\
& \enskip k \in [N - 1].
    \end{aligned}
\end{equation}
We formulate the trajectory optimization problem for UAV systems with mobile charging stations as the following nonlinear program with disjunctive constraints:
\begin{equation}\label{opt: General_OCP}
\begin{array}{ll}
    \underset{\text{ Variables in \eqref{eqn: base var}}}{\text{minimize}}
    &  \sum_{k=1}^{N-1} s_k \\[4pt]
    \;\; \text{subject to}
    & \text{constraints in \eqref{eqn: smooth constr} and \eqref{eqn: disjunctive constr}}.
\end{array}
\end{equation}
The constraints in \eqref{eqn: smooth constr} are compatible with algorithms for smooth nonlinear programs, such as augmented Lagrangian methods. The main challenge in solving the trajectory optimization problem~\eqref{opt: General_OCP} arises from the disjunctive constraints in~\eqref{eqn: disjunctive constr} as they rely on logical OR operations and can induce a disconnected feasible set.  
In Section~\ref{sec: nonlinear_smoothing}, we present a traditional mixed-integer nonlinear formulation and our proposed smooth approximation approach for handling these disjunctive constraints.

\section{Nonlinear Smoothing for Disjunctive Constraints}
\label{sec: nonlinear_smoothing}
% We first summarize the smooth constraints as follows:
% \begin{equation}\label{eqn: smooth constr}
%     \begin{aligned}
% &r_1=\overline{r}_0, \enskip r_N = \overline{r}_f, \enskip e_1=e_{\max},\\
% &\norm{r_{k+1} - r_k}_2 \leq v_{\texttt{d}} s_k,\enskip  k \in [N - 1],\\
% &  e_{\min}\leq e_{k+1}\leq e_{\max},\enskip s_{\min} \leq s_k \le s_{\max},\enskip k \in [N - 1],\\
% & e_{k+1} = f(e_k, s_k, c_k), \enskip k \in [N - 1].
%     \end{aligned}
% \end{equation}
% The disjunctive constraints are as follows:
% \begin{equation}\label{eqn: disjunctive constr}
%     \begin{aligned}
%     &\bigvee_{k \in [N]} \left\{r_k = a_i\right\}, \enskip i \in [\tau], \\
%     & \left(
% \bigvee_{j \in [J]}
% \left\{
% \begin{aligned}
%     c_k &= s_k,\\
%     \varphi_j(r_k) &\leq 0,\\
%     \varphi_j(r_{k+1}) &\leq 0,\\
%     \norm{r_{k+1}-r_k}_2
%     &\leq v_{\texttt{c}} s_k
% \end{aligned}
% \right\}
% \right) \lor \; \{c_k = 0\}, \\
% & \enskip k \in [N - 1].
%     \end{aligned}
% \end{equation}
% The disjunctive constraints in~\eqref{eqn: disjunctive constr} pose unique challenges for optimization, as they rely on logical OR operations and induce a disconnected feasible solution set. 

We first discuss how to model these constraints using discrete variables, leading to a mixed-integer nonlinear programming approach. We then propose an alternative approach that first reformulates the constraints in~\eqref{eqn: disjunctive constr} as nonsmooth constraints and subsequently approximates them with smooth nonlinear functions that are compatible with nonlinear programming.

\subsection{Disjunctive Constraints via Discrete Variables}
A standard approach to model the disjunctions in \eqref{eqn: disjunctive constr} uses binary variables. These disjunctions can be modeled as follows:
\begin{equation}\label{eq: binaryConstraints}
    \begin{aligned}
        & \sum_{k=1}^N U_{ki} = 1, \enskip U_{ki} \in \{0,1\}, \enskip k \in [N], \enskip i \in [\tau],\\ 
        & \|r_k - a_i \|_{\infty} \leq \mu (1 - U_{ki}), \enskip k \in [N], \enskip i \in [\tau],\\
        & \sum_{j=1}^{J + 1} Z_{kj} = 1, \enskip Z_{kj} \in \{0,1\}, \enskip k \in [N - 1], \enskip j \in [J+1], \\
        & |c_k - s_k|\leq \mu (1 - Z_{kj}), \enskip k \in [N - 1], \enskip j \in [J], \\
        & g_j(r_k, r_{k+1}, s_k) \leq  \mu (1 - Z_{kj})\cdot \bm{1}^l,\enskip  k \in [N - 1], \enskip j \in [J], \\ 
        & | c_k |\leq \mu \left(1 - Z_{k(J + 1)}\right), \enskip k \in [N - 1],
    \end{aligned}
\end{equation}
where $\mu \in \mathbb{R}_{>0}$ is a sufficiently large constant. This is the classical Big-M method. The idea is to encode each case of a disjunction with a binary variable and constraints that enforce the case for one value of the binary variable and are trivially satisfied otherwise. Then, constraining the binary variables of each disjunction to sum to one ensures a case is satisfied. In practice, and within our experiments, $\mu$ is set per constraint to the smallest value that makes the constraint redundant when deactivated, as large values weaken the continuous relaxation used by branch-and-bound \cite{camm1990bigm}.

We reformulate the optimization problem~\eqref{opt: General_OCP} as a \emph{mixed-integer nonlinear program (MINLP)} with the following decision variables:
\begin{equation}\label{eqn: MINLP var}
\left\{r_k, e_k, 
        \{U_{k i}\}_{i=1}^{\tau}
    \right\}_{k=1}^N \cup \left\{
        s_k, c_k, \{Z_{kj}\}_{j=1}^{J+1}
    \right\}_{k=1}^{N-1}.
\end{equation}
The resulting MINLP is given by
\begin{equation}\label{opt: MINLP}
\begin{array}{ll}
    \underset{\text{ Variables in \eqref{eqn: MINLP var}}}{\text{minimize}}
    &  \sum_{k=1}^{N-1} s_k \\[4pt]
    \;\; \text{subject to}
    & \text{constraints in \eqref{eqn: smooth constr} and \eqref{eq: binaryConstraints}}.
\end{array}
\end{equation}
We can solve the MINLP above using branch-and-bound–based methods combined with nonlinear programming algorithms. For details on models, algorithms, and practical solution methods for MINLP, we refer interested readers to~\cite{grossmann1997mixed}.

\subsection{Disjunctive Constraints via Nonlinear Smoothing}
One limitation of the discrete-variable approach is that it leads to an exponential growth in the number of possible values for binary variables, which often makes scalable real-time solutions impractical. 
As an alternative, we introduce a continuous modeling approach for the disjunctive constraints in~\eqref{eqn: disjunctive constr} that avoids the use of discrete variables. To this end, we first reformulate the constraints in~\eqref{eqn: disjunctive constr} as follows: 
\begin{equation}\label{eqn: nonsmooth constr}
\begin{aligned}
    &\underset{k \in [N]}{\min} \norm{\begin{bmatrix}
        r_k - a_i\\
        \varepsilon
    \end{bmatrix}}_2 = \varepsilon, \enskip i \in [\tau], \\
    &\min_{j\in[J]} \left\{
    \norm{\begin{bmatrix}
        c_k \\
        \varepsilon
    \end{bmatrix}}_2,
    \left\|
    \begin{bmatrix}
        c_k-s_k\\
        \psi_\delta (g_j(r_k, r_{k+1}, s_k))\\
        \varepsilon\\
    \end{bmatrix}
    \right\|_2
\right\}=\varepsilon, \\
&  k \in [N - 1],
\end{aligned}
\end{equation}
where
\begin{equation}
    \textstyle [\psi_\delta(y)]_j = \begin{cases}
        0, & [y]_j \leq -\delta,\\
        \frac{([y]_j + \delta)^2}{4\delta}, & -\delta < [y]_j \leq \delta,\\
        [y]_j, & [y]_j> \delta.
    \end{cases}
\end{equation}
Here, \([y]_j\) denotes the \(j\)-th entry of the vector \(y\), \(\delta > 0\) is a small smoothing parameter (e.g., \(10^{-3}\)), and \(\varepsilon > 0\) is a tunable parameter.
The key idea is to reformulate each disjunctive constraint using pointwise minimum of the violation of candidate conditions. Here, the function \(\psi_\delta\) provides a smooth measure of the violation of inequality constraint.

Note that the constraints in \eqref{eqn: nonsmooth constr} are not compatible with algorithms for smooth nonlinear programs, due to the nonsmooth pointwise minimum function. 
A common approach to approximate the pointwise minimum function is via the \emph{log-sum-exp} function. However, the log-sum-exp function often causes numerical instabilities due to the rapid growth of the exponential function~\cite{uzun2024optimization}.

To overcome this limitation, we propose a smooth function based on the \(\ell_p\)-norm. In particular, for the nonsmooth minimum \(\min_{i\in[n]} z_i\), where \(z_i\geq \epsilon>0\) for all \(i\in[n]\), we propose the following smooth approximation:
\begin{equation} \label{eqn: proposed_min}
    \phi_p(z)\coloneqq \left(\sum_{i=1}^n z_i^{-p}\right)^{-1/p},
    \enskip p>0.
\end{equation}
The following theorem provides bounds on both the value and gradient direction of the smooth approximation~\eqref{eqn: proposed_min}.

\begin{theorem} \label{theorem: Lp-norm}
    Let \(z\in\mathbb{R}_{>0}^n\), \(z_{\min}=\min_{i\in[n]} z_i\), \(A_{\min}(z)=\{j|j\in[n], z_j=z_{\min}\}\), and \(\Delta= -z_{\min}+\underset{j\in[n], j\notin A_{\min}(z)}{\min} z_j\). Let \(m\) denote the cardinality of set \(A_{\min}(z)\). Let \(v^\star\in\mathbb{R}^n\) be such that
    \begin{equation}\label{eqn: nominal gradient}
        v^\star_i=\begin{cases}
        \frac{1}{m}, & \text{if } i\in A_{\min}(z),\\
        0, & \text{otherwise.}
    \end{cases}
    \end{equation}
    Then 
    \begin{equation}\label{eqn: softmin function value}
        \exp\left(\frac{m-n}{pm \left(1+\Delta/z_{\min}\right)^p}\right)z_{\min}\leq m^{\frac{1}{p}}\phi_p(z)\leq z_{\min}.
    \end{equation}
    Furthermore,
    \begin{equation}\label{eqn: softmin gradient direction}
         \frac{\langle \nabla \phi_p(z), v^\star\rangle}{\norm{\nabla \phi_p(z)}_2\norm{v^\star}_2}\geq \exp\left(\frac{m-n}{2m \left(1+\Delta/z_{\min}\right)^{2p+2}}\right).
    \end{equation}
\end{theorem}
\begin{proof}
    See Appendix.
\end{proof}
%[What happened to \(\Delta\) if all \(z = z_{\min}\)?] 
Here, \(\Delta\) is the gap between the smallest and the second-smallest values, and \(v^\star\) is a Clarke generalized gradient~\cite{clarke1990optimization}. 
%The proof of Theorem~\ref{theorem: Lp-norm} is provided in Appendix~\ref{appendix:proof}. 
Theorem~\ref{theorem: Lp-norm} shows that the accuracy of the approximation in \eqref{eqn: proposed_min}---in both value and gradient direction, measured by the cosine function---increases with the value of $(1+\Delta/z_{\min})^p$, bringing both exponential terms in~\eqref{eqn: softmin function value} and~\eqref{eqn: softmin gradient direction} close to~1.

%The approximation accuracy with the value of $(1+\Delta/z_{\min})^p$.
%Note that the ratio \(\Delta/z_{\min}\) is influenced by the smoothing parameter \(\varepsilon\) introduced in \eqref{eqn: nonsmooth constr}.

We approximate the nonsmooth constraints in \eqref{eqn: nonsmooth constr} using the smooth minimum function proposed in \eqref{eqn: proposed_min}. Thus, we can reformulate the trajectory optimization problem~\eqref{opt: General_OCP} as a smooth \emph{nonlinear program (NLP)}:
\begin{equation}\label{opt: NLP}
\begin{array}{ll}
\underset{\text{Variables in \eqref{eqn: base var}}}{\text{minimize}}
    &  \sum_{k=1}^{N-1} s_k \\
    \;\; \text{subject to}
    & \text{constraints in \eqref{eqn: smooth constr} and smooth approx. } \\
    & \text{of the constraints in \eqref{eqn: nonsmooth constr} via \eqref{eqn: proposed_min}.}
\end{array}
\end{equation}
Note that all of the variables in optimization~\eqref{opt: NLP} are continuous (\emph{i.e.}, no discrete-valued variables) and all of the functions that appear in optimization~\eqref{opt: NLP} are differentiable. 

The nonlinear program~\eqref{opt: NLP} can be difficult to solve due to the choice of smoothing parameters. Specifically, small values of $\varepsilon$ and large values of $p$ sharpen the smooth minimum approximation~\eqref{eqn: proposed_min} but can cause rapid changes in the gradient near ties, rendering the resulting NLP numerically ill-conditioned.
% As a result, in principle one can solve this optimization problem using standard algorithms for nonlinear programs. However, small values of $\varepsilon$ and large values of $p$ sharpen the smooth minimum approximation~\eqref{eqn: proposed_min} but can make the resulting NLP more difficult to solve. 

To address this challenge, we propose a \emph{homotopy method}, a numerical continuation scheme that gradually varies the smoothing parameters to iteratively refine the approximation~\cite{malyuta2023fast}.
% A similar homotopy strategy has been used to handle logical constraints through smooth approximations~\cite{malyuta2023fast}.
% In practice, we observe that the [augmented Lagrangian method] combined with homotopy provides consistently robust performance. 
We summarize the homotopy method for solving optimization problem \eqref{opt: NLP}, where $\bm{x}$ denotes the column vector formed by stacking the decision variables in \eqref{eqn: base var}.
%, and let $\bm{x}^{0}$ denote its initial guess. We gradually decrease $\varepsilon$ and increase $p$, as summarized in Algorithm~\ref{alg:homotopy}.
%Specifically, we initialize the first homotopy stage with initial trajectory $\bm{x}^{0}$. 
At each stage of the homotopy method, we solve \eqref{opt: NLP} with the current $\varepsilon$ and $p$ using an NLP algorithm (\eg, the augmented Lagrangian method~\cite{chen2020convergence}) until the residuals satisfy the tolerances or the iteration count reaches its limit. We then use the resulting trajectory to initialize the next homotopy stage.

% using the augmented Lagrangian method~\cite{chen2020convergence} until the Karush--Kuhn--Tucker (KKT) residuals satisfy the tolerances or the iteration count reaches its limit. We refer to~\cite{haeser2018second} for a definition of the KKT residual. We then use the resulting trajectory to initialize the next homotopy stage.
\begin{algorithm}[t]
\caption{Homotopy Method}
\label{alg:homotopy}
\begin{algorithmic}[1]
\Require Initial trajectory $\bm{x}^{0}$; initial smoothing parameters
$\varepsilon^{0}, p^{0}$; update factors $\alpha>1$, $\beta\in(0,1)$;
bounds $\varepsilon_{\min}, p_{\max}$; number of homotopy stages $L$
\State $\varepsilon \gets \varepsilon^{0}$, \quad
       $p \gets p^{0}$, \quad $\bm{x} \gets \bm{x}^{0}$
\For{$\ell = 1,\dots,L$}
    % \State Solve \eqref{opt: NLP} with the current $\varepsilon$ and $p$ using the some nonlinear algorithms augmented Lagrangian method initialized at $\bm{x}$
    \State Solve \eqref{opt: NLP} with the approximation based on current values of \(\varepsilon\) and \(p\), obtain trajectory $\bm{x}$.
    %using an NLP algorithm initialized at $\bm{x}$
   % \State $\bm{x} \gets$ the resulting trajectory
    \If{$\ell<L$}
        \State $\varepsilon \gets \max\{\beta\varepsilon, \varepsilon_{\min}\}, \, p \gets \min\{\alpha p, p_{\max}\}$
    \EndIf
\EndFor
\State \Return $\bm{x}^{\star} \gets \bm{x}$
\end{algorithmic}
\end{algorithm}

\section{Numerical Simulations}
\label{sec:experiments}

We demonstrate the proposed model through a UAV trajectory-optimization problem with mobile charging support. This problem includes nonlinear constant-current--constant-voltage charging dynamics for the UAV, multiple spatially distributed UAV tasks, and terrain-access constraints on where charging can occur. 
We compare the proposed NLP model with the MINLP model to evaluate the computational scalability of the proposed model.

\subsection{Problem Setup}

We consider a problem with \(d=2\), representing the UAV's position in the \(xy\)-plane. We assume that the UAV operates at a constant altitude and that the time required for vertical motion (\emph{e.g.}, during takeoff and landing) is negligible compared with that spent in horizontal motion.
We consider \(J=4\) circular mobile charging regions, as shown in Fig.~\ref{fig: an_example_solution}. 
The function \(g_j\) in~\eqref{eq:charging_disjunction} is defined as
\begin{equation}\label{eq: defChargingCondition}
    g_j(r_k, r_{k+1}, s_k) \coloneqq \begin{bmatrix}
        \norm{r_k - o_j}_2 - \rho_j\\
        \norm{r_{k+1} - o_j}_2 - \rho_j \\
        \norm{r_k - r_{k+1}}_2 - v_\texttt{c} s_k
    \end{bmatrix},
\end{equation}
where \(o_j\) and \(\rho_j\) denote the center and radius of the charging region for station \(j\), respectively, and \(v_{\texttt{c}}\) denotes the station's maximum speed. The constraint \(g_j(r_k,r_{k+1},s_k)\leq\bm{0}^3\) ensures that both endpoints lie within the same charging region and that the distance traveled does not exceed \(v_{\texttt{c}}s_k\). Assuming straight-line motion between consecutive positions, the entire segment remains within the region because the region is convex.

% Note that the proposed model~\eqref{eq:charging_disjunction} accommodates more general regions defined by differentiable inequality constraints.

In the simulations, we adopt a constant-current--constant-voltage (CC--CV) charging model~\cite{hussein2011review}.
Specifically, we approximate the battery-level evolution using a linear function of time during the CC phase and an exponential function of time during the CV phase.
Let $e_{\mathrm{th}} \in \mathbb{R}_{>0}$ denote the battery level at which charging transitions from the CC phase to the CV phase, and let $\kappa \in \mathbb{R}_{>0}$ denote the charging rate during the CC phase. We define the CV charging function as follows:
\begin{equation}
\gamma(e,t)
\coloneqq e_{\mathrm{max}}-(e_{\mathrm{max}}-e)\exp(-t/\sigma),
\end{equation}
where $\sigma=(e_{\mathrm{max}}-e_{\mathrm{th}})/\kappa$ is the charging time constant chosen to ensure slope continuity at the CC--CV transition.
For $e_k<e_{\mathrm{th}}$, let $\tau_k\coloneqq(e_{\mathrm{th}}-e_k)/\kappa$ denote the time required to reach the charging threshold. Depending on the initial battery level $e_k$ and the charging duration $c_k$, the charging dynamics comprise three cases:
\begin{equation}
\label{eqn:charging_dynamic}
f_{\texttt{c}}(e_k,c_k)=
\begin{cases}
e_k+\kappa c_k,
& e_k<e_{\mathrm{th}},\ c_k\leq\tau_k,\\
\gamma(e_{\mathrm{th}},c_k-\tau_k),
& e_k<e_{\mathrm{th}},\ c_k>\tau_k,\\
\gamma(e_k,c_k),
& e_k\geq e_{\mathrm{th}}.
\end{cases}
\end{equation}
Furthermore, we define the discharging dynamics as 
\begin{equation}\label{eqn: discharging_dynamics}
f_{\texttt{d}}(e_k, s_k - c_k) = e_k - \zeta(s_k - c_k),
\end{equation}
% {\color{red} add slope here}
where \(\zeta \in \mathbb{R}_{>0}\) is the discharging rate. Note that we choose these terrain-access constraints and charging and discharging models for the simulations. The formulation~\eqref{eq:charging_disjunction} accommodates more general regions defined by other convex differentiable inequality constraints, \eg, convex polygons, while the battery dynamics~\eqref{eq:battery_dynamics} accommodate other nonlinear charging and discharging models that satisfy~\eqref{eqn:zero_duration_identity}.

We assume the UAV travels from the initial position $\overline{r}_0$ to the final position $\overline{r}_f$, as illustrated in Fig.~\ref{fig: an_example_solution}.
We set $e_{\max}=1$, $e_{\mathrm{th}}=0.7$, 
$v_{\texttt{d}}=36\,\mathrm{km/h}$, $v_{\texttt{c}}=10.8\,\mathrm{km/h}$, $\kappa=4.625$, and \(\zeta = 2.5\).
Additionally, we set $e_{\min}=0$, $s_{\min}=30\,\mathrm{s}$, and $s_{\max}=1\,\mathrm{h}$ for the inequality constraints in~\eqref{eqn: smooth constr}.
In Algorithm~\ref{alg:homotopy}, we set the initial smoothing parameters to \(\varepsilon^0 = 0.2\) and \(p^0 = 2\), the update factors to \(\alpha = 1.2\) and \(\beta = 0.6\), and the bounds to \(\varepsilon_{\min} = 5 \times 10^{-4}\) and \(p_{\max} = 12\).
We perform all simulations on the Minnesota Supercomputing Institute cluster (\url{https://www.msi.umn.edu/}). Each simulation uses one AMD EPYC 7763 CPU core and 12~GB of allocated memory.

\subsection{Comparison against MINLP}

We solve the proposed NLP model in~\eqref{opt: NLP} using Algorithm~\ref{alg:homotopy}. Specifically, within each homotopy stage, we apply the augmented Lagrangian method~\cite{chen2020convergence}, using \texttt{L-BFGS} (implemented in \texttt{fminunc} in \texttt{MATLAB}) to minimize the augmented Lagrangian at each iteration (see~\cite{doi:10.1137/0312021}). We terminate the augmented Lagrangian iterations when the Karush--Kuhn--Tucker (KKT) residuals satisfy the prescribed tolerances or the maximum number of iterations is reached. We refer to~\cite{haeser2018second} for the definition of the KKT residuals. The resulting trajectory is then used to initialize the next homotopy stage.

We provide two MINLP implementations, each based on \eqref{opt: MINLP}. Both are in the \texttt{JuMP}~\cite{JuMP} framework.
For the first, we use the open-source solver \texttt{Juniper}~\cite{Juniper}, which implements a branch-and-bound algorithm. We configure it to use \texttt{IPOPT}~\cite{IPOPT} as the inner NLP solver (with a
limited-memory Hessian approximation and a convergence tolerance of $10^{-4}$).
For the second, we use the open-source solver \texttt{SCIP}, which employs spatial branch-and-bound for global optimization~\cite{SCIPOptSuite10}. We configure \texttt{SCIP} with a constraint feasibility tolerance of $10^{-5}$. Additionally, we set the relative optimality gap to $10^{-2}$ for \texttt{Juniper} and $10^{-3}$ for \texttt{SCIP}. All other non-default solver settings can be found in Appendix~\ref{appendix:solver_settings}.
To accommodate the piecewise charging curves in~\eqref{eqn:charging_dynamic} within the spatial branch-and-bound framework, we introduce additional binary variables. The resulting reformulation of~\eqref{eq:battery_dynamics} is given in Appendix~\ref{appendix:formulation}, and is used within our \texttt{SCIP} implementation. Additionally, we note that the definition of $g_j$ in \eqref{eq: defChargingCondition}
enables an equivalent formulation of \eqref{eq: binaryConstraints} involving fewer binary constraints. This reformulation is described in Appendix~\ref{appendix: binaryConstraints}, and is used within both of the MINLP implementations.

The convergence of NLP algorithms is sensitive to initialization~\cite{betts1998survey,yuan2025filtering}. We warm-start Algorithm~\ref{alg:homotopy} with an initial trajectory \(\bm{x}^0\) that visits the charging station nearest to each task location before visiting that location, then returns to the same station before proceeding to the next task.
For a fair comparison, we warm-start \texttt{Juniper} and \texttt{SCIP} for the MINLP with the same initial trajectory \(\bm{x}^0\) used for the NLP model.

We demonstrate the solution obtained by solving the proposed NLP model~\eqref{opt: NLP} using Algorithm~\ref{alg:homotopy}. Fig.~\ref{fig: an_example_solution} shows the resulting UAV trajectory. Fig.~\ref{fig: uav_battery} shows how the UAV’s battery level and its distance to the closest mobile charging region evolve along this trajectory. The shaded regions indicate periods during which the UAV charges on a mobile charging station. Fig.~\ref{subfigure: history of UAV battery level} illustrates that the proposed model accommodates partial recharging, nonlinear charging dynamics, and multiple task visits within a single discharging cycle.

We compare the performance of the proposed NLP model~\eqref{opt: NLP} with that of the MINLP model~\eqref{opt: MINLP} for different numbers of UAV tasks (denoted by \(\tau\)) with a computation time limit of 10 hours for each case, as shown in Fig.~\ref{fig:three_method_comparison}.
Specifically, for each value of \(\tau\), we consider 100 problem instances with randomly generated UAV task locations. For each method, we plot the solution results for all 100 problem instances, reporting the computation time, objective function value~\eqref{eqn: obj}, and constraint violation. We define the constraint violation using the \(\ell_{\infty}\) norm of the violations of all constraints in~\eqref{eqn: smooth constr} and~\eqref{eqn: disjunctive constr}. 
For the proposed NLP model~\eqref{opt: NLP}, we numerically assess the tightness of the smooth approximation~\eqref{eqn: proposed_min} by evaluating the corresponding exponential terms in~\eqref{eqn: softmin function value} and~\eqref{eqn: softmin gradient direction} in Theorem~\ref{theorem: Lp-norm} at the computed solutions. These terms are nearly~1, with median deviations from~1 across 100 runs ranging from $3.00\times10^{-9}$ to $3.59\times10^{-8}$ across all \(\tau\).
For the MINLP results, we report the objective function value and constraint violation for the best feasible solution found within the time limit, which achieves the lowest objective function value among all feasible solutions found. The proposed NLP model achieves a median computation time below one minute, while the MINLP model often reaches the ten-hour time limit as the number of task locations \(\tau\) increases. The NLP model also achieves objective values and constraint violations comparable to those of the MINLP model.

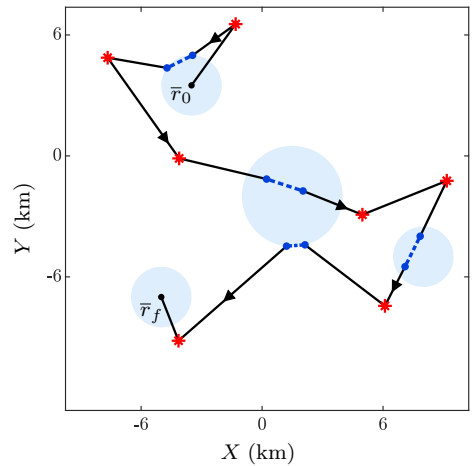
\begin{figure}[!t]
\centering
\begin{tikzpicture}

% include the EPS figure
\node[anchor=south west, inner sep=0] (img) 
    {\includegraphics[width=0.65\linewidth]{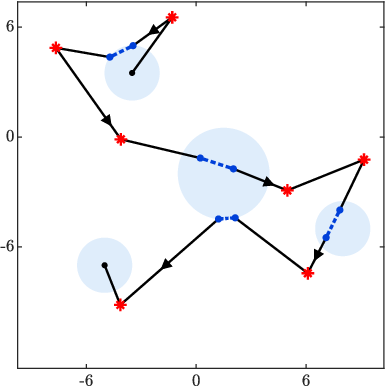}};

% define coordinate system based on image
\begin{scope}[x={(img.south east)}, y={(img.north west)}]

\node at (0.5,-0.05) {\footnotesize $ X\;(\mathrm{km})$};
\node[rotate=90] at (-0.05,0.5) {\footnotesize $ Y\;(\mathrm{km})$};

\node at (0.32,0.79) {\footnotesize $\overline{r}_0$};
\node at (0.25,0.28) {\footnotesize $\overline{r}_f$};

\end{scope}

\end{tikzpicture}
\caption{Solution trajectory obtained using Algorithm~\ref{alg:homotopy} for $\tau = 7$ UAV task locations (red asterisks). Blue circles represent mobile charging regions. Blue dashed lines indicate trajectory segments during which the UAV charges on a mobile station.}
\label{fig: an_example_solution}
\end{figure}

\begin{figure}[!t] 
    \centering
    \begin{subfigure}{\columnwidth}
    \pgfplotsset{every tick label/.append style={font=\scriptsize}}

\begin{tikzpicture}
% \begin{axis}[ymin=0, ymax=0.5,
%      xmin=0, xmax=2,  
%      width=0.95\linewidth,
%      height=0.5\linewidth,
%         xtick = {0, 1, 2},
%         ytick={0,0.2,0.4},
%         yticklabels={0,0.2,0.4},
%         xlabel={\footnotesize Time (h)},
%         ylabel={\footnotesize UAV Remaining T-o-F (h)}
%         ]

\begin{axis}[ymin=0, ymax=1.2,
     xmin=0, xmax=2,  
     width=0.95\linewidth,
     height=0.5\linewidth,
        xtick = {0, 1, 2},
        ytick=  {0, 0.5, 1},
        yticklabels={0, 0.5, 1},
        xlabel={\footnotesize Time (h)},
        ylabel={\footnotesize UAV Battery Level}
        ]

    % \addplot[solid, very thick, color = black, line join=round]
    %   table[x=time, y=t_A] {battery_refined.dat};

    \addplot[solid, very thick, color=black, line join=round]
    table[x=time, y expr={\thisrow{t_A}/0.4}]
    {battery_refined.dat};

      % AOI visitation points: (time, battery level)
      \addplot[ only marks, mark=asterisk, mark options={draw=red, line width=0.8pt, scale=1.2} ] coordinates {
        (0.104118083957826,0.739704790105435)
        (0.414823306242361,0.7460539871142925)
        (0.586775866257890,0.316172587075470)
        (0.977683013815450,0.5760502166269175)
        (1.10380627963211,0.260742052085265)
        (1.40601626564344,0.610987016315625)
        (1.82849233974556,0.163717488864610)
    };

    % \addplot[gray, dashed, thick, forget plot] coordinates {(\pgfkeysvalueof{/pgfplots/xmin},0.0) (\pgfkeysvalueof{/pgfplots/xmax},0.0)};

\addplot[
draw=none,
fill=gray,
fill opacity=0.25
] coordinates {
(0.177430388465736,\pgfkeysvalueof{/pgfplots/ymin})
(0.331306973516325,\pgfkeysvalueof{/pgfplots/ymin})
(0.331306973516325,\pgfkeysvalueof{/pgfplots/ymax})
(0.177430388465736,\pgfkeysvalueof{/pgfplots/ymax})
};

% interval 2
\addplot[draw=none, fill=gray, fill opacity=0.25] coordinates {
(0.712216345882417,\pgfkeysvalueof{/pgfplots/ymin})
(0.888772035875738,\pgfkeysvalueof{/pgfplots/ymin})
(0.888772035875738,\pgfkeysvalueof{/pgfplots/ymax})
(0.712216345882417,\pgfkeysvalueof{/pgfplots/ymax})
};

% interval 3
\addplot[draw=none, fill=gray, fill opacity=0.25] coordinates {
(1.18874723633867,\pgfkeysvalueof{/pgfplots/ymin})
(1.34504190686512,\pgfkeysvalueof{/pgfplots/ymin})
(1.34504190686512,\pgfkeysvalueof{/pgfplots/ymax})
(1.18874723633867,\pgfkeysvalueof{/pgfplots/ymax})
};

% interval 4
\addplot[draw=none, fill=gray, fill opacity=0.25] coordinates {
(1.54489488115274,\pgfkeysvalueof{/pgfplots/ymin})
(1.63035742961651,\pgfkeysvalueof{/pgfplots/ymin})
(1.63035742961651,\pgfkeysvalueof{/pgfplots/ymax})
(1.54489488115274,\pgfkeysvalueof{/pgfplots/ymax})
};

\end{axis}
\end{tikzpicture}
    \caption{History of UAV battery level.}
    \label{subfigure: history of UAV battery level}
    \end{subfigure}

    \bigskip

    \begin{subfigure}{\columnwidth}
    \pgfplotsset{every tick label/.append style={font=\scriptsize}}

\begin{tikzpicture}
\begin{axis}[ymin=-0.5, ymax=4,
     xmin=0, xmax=2,  
     width=0.95\linewidth,
     height=0.5\linewidth,
        xtick = {0, 1, 2},
        ytick = {0, 4.0},
        yticklabels={0, 4.0},
        xlabel={\footnotesize Time (h)},
        ylabel={\footnotesize Distance (km)}
        ]

    \addplot[solid, very thick, color = black, line join=round]
      table[x=time, y=d] {distance_history.dat};

    \addplot[gray, dashed, thick, forget plot] coordinates {(\pgfkeysvalueof{/pgfplots/xmin},0) (\pgfkeysvalueof{/pgfplots/xmax},0)};
        
    % median sample points
    \addplot[
      only marks,
      mark=asterisk,
      mark options={draw=red, line width=0.8pt, scale=1.2}
    ]
    coordinates {
    (0.104118083957826,2.24500951373025)
    (0.414823306242361,2.88039620523246)
    (0.586775866257890,2.17473727488609)
    (0.977683013815450,1.09444596348982)
    (1.10380627963211,2.44113677534516)
    (1.40601626564344,1.59456968260571)
    (1.82849233974556,0.829367573602537)

    };

\addplot[
draw=none,
fill=gray,
fill opacity=0.25
] coordinates {
(0.177430388465736,\pgfkeysvalueof{/pgfplots/ymin})
(0.331306973516325,\pgfkeysvalueof{/pgfplots/ymin})
(0.331306973516325,\pgfkeysvalueof{/pgfplots/ymax})
(0.177430388465736,\pgfkeysvalueof{/pgfplots/ymax})
};

% interval 2
\addplot[draw=none, fill=gray, fill opacity=0.25] coordinates {
(0.712216345882417,\pgfkeysvalueof{/pgfplots/ymin})
(0.888772035875738,\pgfkeysvalueof{/pgfplots/ymin})
(0.888772035875738,\pgfkeysvalueof{/pgfplots/ymax})
(0.712216345882417,\pgfkeysvalueof{/pgfplots/ymax})
};

% interval 3
\addplot[draw=none, fill=gray, fill opacity=0.25] coordinates {
(1.18874723633867,\pgfkeysvalueof{/pgfplots/ymin})
(1.34504190686512,\pgfkeysvalueof{/pgfplots/ymin})
(1.34504190686512,\pgfkeysvalueof{/pgfplots/ymax})
(1.18874723633867,\pgfkeysvalueof{/pgfplots/ymax})
};

% interval 4
\addplot[draw=none, fill=gray, fill opacity=0.25] coordinates {
(1.54489488115274,\pgfkeysvalueof{/pgfplots/ymin})
(1.63035742961651,\pgfkeysvalueof{/pgfplots/ymin})
(1.63035742961651,\pgfkeysvalueof{/pgfplots/ymax})
(1.54489488115274,\pgfkeysvalueof{/pgfplots/ymax})
};

\end{axis}
\end{tikzpicture}
    \caption{History of the distance from the UAV to its closest mobile charging region.}
    \end{subfigure}
\caption{Histories of the UAV's battery level and distance to its closest mobile charging region for the solution trajectory shown in Fig.~\ref{fig: an_example_solution}. Red asterisks indicate visits to UAV task locations.}
\label{fig: uav_battery}
\end{figure}
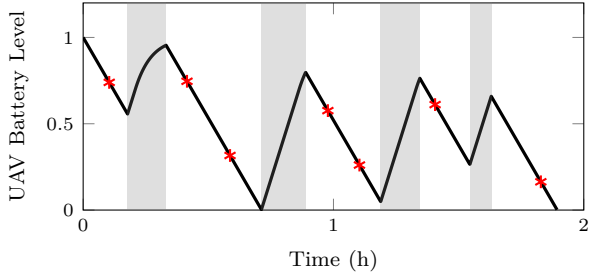
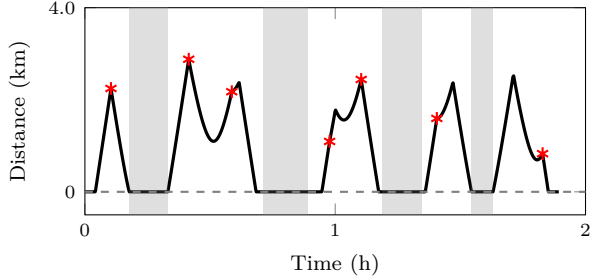

\begin{figure}[!t]
    \centering
    \includegraphics[width=\linewidth]{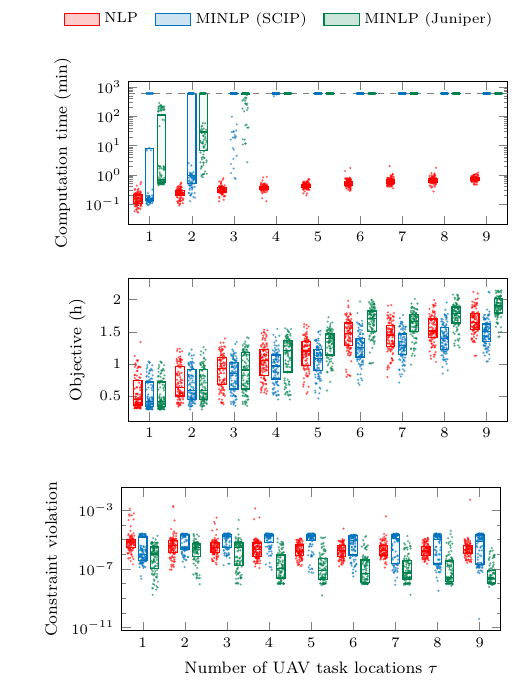}

    \caption{Comparison of NLP (red, left) using
    Algorithm~\ref{alg:homotopy}, MINLP solved using
    \texttt{SCIP} (blue, middle), and MINLP solved using
    \texttt{Juniper} (green, right), with a ten-hour time
    limit per MINLP run, across 100 Monte Carlo instances
    for each number of UAV task locations \(\tau\).
    Dots represent individual samples. Boxes show the
    25th--75th percentiles, with medians marked inside.
    The dashed line indicates the ten-hour time limit.}
    \label{fig:three_method_comparison}
\end{figure}
\section{Conclusion}

We presented a nonlinear trajectory optimization model for UAV systems with mobile charging support. This model includes nonlinear charging dynamics for the UAV, multiple spatially distributed UAV tasks, and terrain-access constraints on mobile charging stations. We based this model on a smoothing approximation of disjunctive constraints, which eliminates the need for integer programming. Compared with mixed-integer nonlinear programming, this model reduces computation time from hours to minutes in numerical simulations. In future work, we plan to extend the current model to multi-UAV systems with mobile charging support and persistent monitoring applications.

%\addtolength{\textheight}{-12cm}   % This command serves to balance the column lengths
                                  % on the last page of the document manually. It shortens
                                  % the textheight of the last page by a suitable amount.
                                  % This command does not take effect until the next page
                                  % so it should come on the page before the last. Make
                                  % sure that you do not shorten the textheight too much.

%%%%%%%%%%%%%%%%%%%%%%%%%%%%%%%%%%%%%%%%%%%%%%%%%%%%%%%%%%%%%%%%%%%%%%%%%%%%%%%%

\section*{APPENDIX}

\subsection{Proof of Theorem~\ref{theorem: Lp-norm}:}
\label{appendix:proof}
% \noindent\textit{Proof of Theorem~\ref{theorem: Lp-norm}:}
\begin{proof}
We first prove the bound in \eqref{eqn: softmin function value}. Observe that
\begin{equation}\label{eqn: phi func}
\begin{aligned} &\phi_p(z)=\textstyle \left(\sum_{i\in A_{\min}(z)} z_i^{-p}+\sum_{j\notin A_{\min}(z)} z_j^{-p}\right)^{-\frac{1}{p}}\\
 &= m^{-\frac{1}{p}}z_{\min} \left(1+q\right)^{-\frac{1}{p}},
\end{aligned}
\end{equation}  
where \(q=\frac{1}{m}\sum_{j\notin A_{\min}(z)} \left(\frac{z_j}{z_{\min}}\right)^{-p}\). Since \(q\geq 0\) and \(z_j\geq z_{\min}+\Delta\) for all \(j\notin A_{\min}(z)\), we can show that
\begin{equation}\label{eqn: q bound}
    0\leq q\leq \frac{n-m}{m} \left(1+\Delta/z_{\min}\right)^{-p}.
\end{equation}
By substituting the first inequality in \eqref{eqn: q bound} into \eqref{eqn: phi func}, we obtain the upper bound in \eqref{eqn: softmin function value}. Furthermore, we can show that
\begin{equation}\label{eqn: exp bound}
    \begin{aligned}
        &\textstyle (1+q)^{-\frac{1}{p}} =\exp\left(-\frac{1}{p}\ln (1+q)\right)\\
        &\textstyle \geq \exp\left(-\frac{1}{p}q\right)\geq \exp\left(\frac{m-n}{pm \left(1+\frac{\Delta}{z_{\min}}\right)^{p}}\right),
    \end{aligned}
\end{equation}
The first inequality in \eqref{eqn: exp bound} holds because \(\ln(1+q)\leq q\) for all \(q>0\) and the exponential function is strictly increasing. The second inequality is due to the upper bound in \eqref{eqn: q bound}. By substituting \eqref{eqn: exp bound} into \eqref{eqn: phi func}, we obtain the lower bound in \eqref{eqn: softmin function value}.  

Next, we prove \eqref{eqn: softmin gradient direction}. Without loss of generality, we assume that \(z_1=z_{\min}\). Let \(g_i\) denote the \(i\)-th element of vector \(\nabla \phi_p(z)\). We can show that \(\norm{v^\star}_2=\frac{1}{\sqrt{m}}\), and \(\langle g, v^\star\rangle = g_1 \). Hence
\begin{equation}\label{eqn: cosine angle}
    \begin{aligned}
       \textstyle \frac{\langle \nabla \phi_p(z), v^\star\rangle}{\norm{\nabla \phi_p(z)}_2\norm{v^\star}_2}=\left(\frac{1}{m}\sum_{i=1}^n \left(g_i/g_1\right)^2\right)^{-\frac{1}{2}}
    \end{aligned}
\end{equation}
Based on the definition of \(\phi_p(z)\), we can show that
\begin{equation}\label{eqn: gradient ratio}
    \begin{aligned}
        & \textstyle \sum_{i=1}^n \left(g_i/g_1\right)^2=m + \sum_{j\notin A_{\min}(z)} \left(z_j/z_{\min}\right)^{-2p-2}\\
        &\textstyle \leq m + (n-m) \left(1+\Delta/z_{\min}\right)^{-2p-2},
    \end{aligned}
\end{equation}
where the last step holds because \(z_j\geq z_{\min}+\Delta\) for all \(j\notin A_{\min}(z)\). Let \(r=\frac{n-m}{m \left(1+\Delta/z_{\min}\right)^{2p+2}}\). By substituting \eqref{eqn: gradient ratio} into \eqref{eqn: cosine angle}, we can show that
\begin{equation}
    \begin{aligned}
        \textstyle \frac{\langle \nabla \phi_p(z), v^\star\rangle}{\norm{\nabla \phi_p(z)}_2\norm{v^\star}_2} \geq \left(1+r\right)^{-\frac{1}{2}}=\exp\left(-\frac{1}{2}\ln(1+r)\right)
    \end{aligned}
\end{equation}
Since \(\ln(1+r)\leq r\) for all \(r>0\), we conclude that  \(\frac{\langle \nabla \phi_p(z), v^\star\rangle}{\norm{\nabla \phi_p(z)}_2\norm{v^\star}_2}\geq \exp\left(-\frac{1}{2}r\right) \), which is exactly \eqref{eqn: softmin gradient direction}.
\end{proof}

\subsection{MINLP Experimental Notes}

\subsubsection{Solver Settings}
\label{appendix:solver_settings}
In Table~\ref{tab:solver_settings}, all non-default solver settings for the MINLP implementations are listed. In addition, we impose a computation time limit of 10 hours for all solvers. All other settings are configured to their default.
\begin{table}
\centering
\caption{Non-default solver settings for the MINLP implementations.}
\label{tab:solver_settings}
\footnotesize
\renewcommand{\arraystretch}{1.1}
\setlength{\tabcolsep}{4pt}
\begin{tabular}{@{}lll@{}}
\hline
\textbf{Solver} & \textbf{Option} & \textbf{Value} \\
\hline
\texttt{Juniper} & \texttt{mip\_gap} & $10^{-2}$ \\
\hline
\texttt{IPOPT} & \texttt{hessian\_approximation} & \texttt{limited-memory} \\
 & \texttt{tol} & $10^{-4}$ \\
 & \texttt{acceptable\_tol} & $10^{-3}$ \\
 & \texttt{acceptable\_iter} & $5$ \\
 & \texttt{max\_iter} & $1000$ \\
 & \texttt{mu\_strategy} & \texttt{adaptive} \\
 & \texttt{expect\_infeasible\_problem} & \texttt{yes} \\
\hline
\texttt{SCIP} & \texttt{numerics/feastol} & $10^{-5}$ \\
 & \texttt{limits/gap} & $10^{-3}$ \\
\hline
\end{tabular}
\end{table}

\subsubsection{Reformulation of the Battery Dynamics for \texttt{SCIP}}
\label{appendix:formulation}
An alternative mixed-integer nonlinear formulation of the battery dynamics~\eqref{eq:battery_dynamics}, incorporating the charging dynamics~\eqref{eqn:charging_dynamic} and the discharging dynamics~\eqref{eqn: discharging_dynamics}, is given by
\begin{equation}\label{eqn:charging_branch}
    \begin{aligned}
        & \sum_{p=1}^{3} P_{kp} = 1 - Z_{k(J+1)}, \enskip P_{kp} \in \{0,1\},
         \enskip k \in [N-1],\ p \in [3], \\
        & e_k + \kappa c_k - e_{\mathrm{th}} \leq \mu (1 - P_{k1}),
             \enskip k \in [N-1], \\
        & \begin{bmatrix} e_k - e_{\mathrm{th}} \\ e_{\mathrm{th}} - e_k - \kappa c_k \end{bmatrix}
            \leq \mu (1 - P_{k2})\cdot\bm{1},
             \enskip k \in [N-1], \\
        & e_{\mathrm{th}} - e_k \leq \mu (1 - P_{k3}),
             \enskip k \in [N-1], \\
        & | e_{k+1} - g_{\mathrm{c}}^{p}(e_k, c_k) |
            \leq \mu (1 - P_{kp}),
             \enskip k \in [N-1],\ p \in [3], \\
        & | e_{k+1} - g_{\mathrm{d}}(e_k, s_k) |
            \leq \mu (1 - Z_{k(J+1)}),
             \enskip k \in [N-1],
    \end{aligned}
\end{equation}
where $g_{\mathrm{c}}^{p}$ denotes the $p$-th case of \eqref{eqn:charging_dynamic}, $P_{kp}$ is a binary variable indicating whether that case is selected at step $k$, and $\mu$ is a sufficiently large positive constant. These big-$M$ constraints select exactly one charging case when $Z_{k(J+1)}=0$ and enforce the discharge dynamics when $Z_{k(J+1)}=1$. In the latter case, \eqref{eq: binaryConstraints} also enforces $c_k=0$.

\subsubsection{Elimination of Redundant Constraints in \eqref{eq: binaryConstraints}}
\label{appendix: binaryConstraints}
Among the binary constraints in \eqref{eq: binaryConstraints} are
\begin{equation}\label{eq: chargingSpatialCondition}
    \begin{aligned}
        & g_j(r_k, r_{k+1}, s_k) \leq  \mu (1 - Z_{kj})\cdot \bm{1}^l,\enskip  k \in [N - 1], \enskip j \in [J].
    \end{aligned}
\end{equation} In \eqref{eq: defChargingCondition}, we define the function $g_j$ used within our experiments. Notably, this function includes the term $\norm{r_k - r_{k+1}}_2 - v_\texttt{c} s_k$, which is identical across all charging regions $j\in[J]$. In light of this, we can write \eqref{eq: chargingSpatialCondition} by constraining the following for all $k\in[N-1]$:
\begin{equation}\label{eq: lessRedundantCharging}
    \begin{aligned}
        & \begin{bmatrix}
       \norm{r_k - o_j}_2 - \rho_j  \\\
        \norm{r_{k+1} - o_j}_2 - \rho_j \\
        \end{bmatrix} \leq  \mu (1 - Z_{kj})\cdot\begin{bmatrix}
        1\\
        1
        \end{bmatrix}, \enskip j \in [J], \\ 
        & \|{r_{k+1} - r_k}\|_2 \leq v_{\texttt{c}}s_k + \mu Z_{k(J+1)},
    \end{aligned}
\end{equation}
where $\mu \in \mathbb{R}_{>0}$ is a sufficiently large constant. The idea is to enforce the constraint on the last line when any charging condition is selected---which occurs when $Z_{k(J+1)} = 0$---without introducing $J$ constraints per time step $k\in[N-1]$. Since the above equation is equivalent to \eqref{eq: chargingSpatialCondition} while eliminating redundant constraints, we opt to use it within our mixed-integer experiments.

% % \section*{APPENDIX}
% \section*{ACKNOWLEDGMENT} Some acknowledgements.

\bibliographystyle{IEEEtran}
\bibliography{reference_ACC}

\end{document}